\documentclass[11pt,reqno]{amsart}

\newtheorem{proposition}{Proposition}[section]

\newtheorem{corollary}[proposition]{Corollary}
\newtheorem{theorem}[proposition]{Theorem}

\theoremstyle{definition}

\newtheorem{example}[proposition]{Example}

\newtheorem{remark}[proposition]{Remark}

\newcommand{\thlabel}[1]{\label{th:#1}}
\newcommand{\thref}[1]{Theorem~\ref{th:#1}}
\newcommand{\selabel}[1]{\label{se:#1}}

\newcommand{\prlabel}[1]{\label{pr:#1}}
\newcommand{\prref}[1]{Proposition~\ref{pr:#1}}
\newcommand{\colabel}[1]{\label{co:#1}}
\newcommand{\coref}[1]{Corollary~\ref{co:#1}}
\newcommand{\relabel}[1]{\label{re:#1}}

\newcommand{\exlabel}[1]{\label{ex:#1}}

\newcommand{\eqlabel}[1]{\label{eq:#1}}
\newcommand{\equref}[1]{(\ref{eq:#1})}

\def\ot{\otimes}

\def\RR{{\mathbb R}}

\def\CC{{\mathbb C}}

\newcommand{\Cc}{\mathcal{C}}

\def\*C{{}^*\hspace*{-1pt}{\Cc}}
\def\text#1{{\rm {\rm #1}}}

\input xy
\xyoption {all} \CompileMatrices

\usepackage{amssymb}
\usepackage{color,amssymb,graphicx,amscd,amsmath}
\usepackage[colorlinks,urlcolor=blue,linkcolor=blue,citecolor=blue]{hyperref}

\begin{document}
\title[It\^{o}'s theorem and metabelian Leibniz algebras]
{It\^{o}'s theorem and metabelian Leibniz algebras}



\author{A. L. Agore}
\address{Faculty of Engineering, Vrije Universiteit Brussel, Pleinlaan 2, B-1050 Brussels, Belgium}
\email{ana.agore@vub.ac.be and ana.agore@gmail.com}

\author{G. Militaru}
\address{Faculty of Mathematics and Computer Science, University of Bucharest, Str.
Academiei 14, RO-010014 Bucharest 1, Romania}
\email{gigel.militaru@fmi.unibuc.ro and gigel.militaru@gmail.com}

\thanks{A.L. Agore is Postdoctoral Fellow of the Fund for Scientific Research Flanders (Belgium) (F.W.O.–
Vlaanderen). This work was supported by a grant of the Romanian
National Authority for Scientific Research, CNCS-UEFISCDI, grant
no. 88/05.10.2011.}



\maketitle

\begin{abstract}
We prove that the celebrated It\^{o}'s theorem for groups remains
valid at the level of Leibniz algebras: if $\mathfrak{g}$ is a
Leibniz algebra such that $\mathfrak{g} = A + B$, for two abelian
subalgebras $A$ and $B$, then $\mathfrak{g}$ is metabelian, i.e.
$[ \, [\mathfrak{g}, \, \mathfrak{g}], \, [ \mathfrak{g}, \,
\mathfrak{g} ] \, ] = 0$. A structure type theorem for metabelian
Leibniz/Lie algebras is proved. All metabelian Leibniz algebras
having the derived algebra of dimension $1$ are described,
classified and their automorphisms groups are explicitly
determined as subgroups of a semidirect product of groups $P^*
\ltimes \bigl(k^* \times {\rm Aut}_{k} (P) \bigl)$ associated to
any vector space $P$.
\end{abstract}

\section*{Introduction}
An old problem addressed in group theory in the fifties is the
following (\cite[pg. 18]{AFG}): for two given abelian groups $A$
and $B$ describe and classify all groups $G$ which can be written
as a product $G = AB$. Surprisingly, the problem turned out to be
a very difficult one and lead to the development of a vast
literature (see \cite{AFG} and the references therein). However,
the celebrated It\^{o}'s theorem \cite{ito} remains the most
satisfying result obtained so far (cf. \cite{AFG, cond}): if $G =
AB$ is a product of two abelian subgroups, then $G$ is metabelian.
It\^{o}'s theorem provides an important piece of information on
this problem and it was the foundation of many structural results
for finite groups.

The first aim of this paper it to prove that It\^{o}'s theorem
remains valid at the level of Leibniz algebras (\coref{itoLie}):
if $\mathfrak{g}$ is a Leibniz algebra such that $\mathfrak{g} = A
+ B$, for two abelian subalgebras $A$ and $B$, then $\mathfrak{g}$
is metabelian ($2$-step solvable) - the corresponding theorem for
Lie algebras is well known \cite{Baht, Burde, Petra}. The result
is a consequence of a more general statement (\thref{itoLbz})
whose proof, even if is different from It\^{o}'s original argument
used in the case of groups or Lie algebras, it is also based on a
commutator computation involving this time the Leibniz law and a
surprising partial skew-symmetry that exists at the level of the
derived algebra of any Leibniz algebra. A possible generalization
of It\^{o}'s theorem is also discussed: more precisely, we show by
a counterexample that in a certain sense \coref{itoLie} is the
best result we can hope for. Also, we prove again by
counterexamples that the converse of It\^{o}'s theorem does not
hold for Lie algebras nor Leibniz algebras. The study of
metabelian Lie algebras concerns mainly the free metabelian Lie
algebras in connection to algebraic geometry, their automorphisms
groups or Gr\"{o}bner-–Shirshov bases \cite{chen, Dang, dani, Dre,
poro}. On the other hand, the study of Leibniz algebras as
non-commutative generalizations of Lie algebras, is an active
research field arising from homological algebra, classical or
non-commutative differential geometry, vertex operator algebras or
integrable systems. For further details see \cite{am-2013b, ACM,
cam, CK, casa, covez, cu, demir, Dre, fialo, LoP, omirov} and
their references: we just mention that one of the open questions
consists of finding a generalization of Lie's third theorem in the
setting of Leibniz algebras \cite{covez}. Related to It\^{o}'s
theorem for Lie algebras we have to mention another old problem
initiated by Schur in 1905 concerned with finding the abelian
subalgebras of maximal dimension in a finite dimensional Lie
algebra: for recent progress on this problem see \cite{BurdeC,
cebal}.

Turning to the problem we started with, we can state its
counterpart at the level of Lie (resp. Leibniz) algebras:
\emph{for two given abelian algebras $A$ and $B$, describe and
classify all Lie (resp. Leibniz) algebras $\mathfrak{g}$
containing $A$ and $B$ as subalgebras such that $\mathfrak{g} = A
+ B$}. In the finite case, taking into account \coref{itoLie},
this problem is a special case of the classification problem of
all metabelian Lie (resp. Leibniz) algebras. Any metabelian
Lie/Leibniz algebra is a $2$-step solvable Lie/Leibniz algebra.
The classification of all solvable Lie algebras was achieved up to
dimension $4$ over arbitrary fields \cite{gra2} and up to
dimension $6$ over $\RR$ \cite{tur}. On the other hand the
classification of all solvable Leibniz algebras is much more
difficult and is known only up to dimension $4$ over the complex
field $\CC$ \cite{CK, casa, demir}. Thus, it is more likely to be
able to classify metabelian Lie (resp. Leibniz) algebras of a
given dimension rather than solvable ones. The first step towards
this goal is given in \thref{structura} (resp.
\coref{liemetabstr}), where a structure type theorem for
metabelian Leibniz (resp. Lie) algebras is proposed. As an
application, in \thref{clasific1} all metabelian Leibniz algebras
having the derived algebra of dimension $1$ are explicitly
described and classified: there are three families of such Leibniz
algebras. However, in order to write down the isomorphism classes,
it turns out that a more general theory of Kronecker-Williamson's
classification of bilinear forms \cite{sz} needs to be developed.
\coref{automr} offers the detailed description of the
automorphisms groups of these Leibniz algebras: they are subgroups
of a canonical semidirect product of groups $P^* \ltimes \bigl(k^*
\times {\rm Aut}_{k} (P) \bigl)$ that we construct for any vector
space $P$.

\section{It\^{o}'s theorem for Leibniz algebras}\selabel{itolbz}

\subsection*{Notations and terminology}
All vector spaces, (bi)linear maps, Lie or Leibniz algebras are
over an arbitrary field $k$ whose group of units will be denoted
by $k^*$. For a vector space $P$ we denote its dual by $P^* = {\rm
Hom}_k (P, \, k)$ and by ${\rm Aut}_k (P)$ the group of linear
automorphisms of $P$. A map $f: V \to W$ between two vector spaces
is called the trivial map if $f (v) = 0$, for all $v\in V$. A
Leibniz algebra is a vector space $\mathfrak{g}$, together with a
bilinear map $[- , \, -] : \mathfrak{g} \times \mathfrak{g} \to
\mathfrak{g}$ satisfying the Leibniz identity for any $x$, $y$, $z
\in \mathfrak{g}$:
\begin{equation}\eqlabel{Lbz1}
\left[ x,\, \left[y, \, z \right] \right] = \left[ \left[x, \,
y\right], \, z \right] - \left[\left[x, \, z\right] , \, y\right].
\end{equation}
Any Lie algebra is a Leibniz algebra, and a Leibniz algebra
$\mathfrak{g}$ satisfying $[x, \, x] = 0$, for all $x \in
\mathfrak{g}$ is a Lie algebra. Thus, Leibniz algebras are non
skew-symmetric generalizations of Lie algebras. However, for any
Leibniz algebra $\mathfrak{g}$ there exists a partial
skew-symmetry on the derived algebra. More precisely, the
following formula which will be used in the proof of
\thref{itoLbz}, holds for any $x$, $y$, $z$, $t \in \mathfrak{g}$:
\begin{equation}\eqlabel{surpriza}
\left[ \, \left[x,\, y \right], \, \left[z, \, t \right] \,
\right] = - \left[\, \left[ x, \, y \right], \, \left[t, \, z
\right] \, \right].
\end{equation}
Indeed, the equality \equref{surpriza} follows by applying two
times the Leibniz law as follows:
$$
\left[ \, \left[ x, \, y \right], \, \left[ z, \, t \right] \,
\right] \stackrel{\equref{Lbz1}}{=} \left[ \, \left[ \, \left[ x,
\, y \right], \, z \right], \, t \right] - \left[ \, \left[ \,
\left[ x, \, y \right], \, t \right], \, z \right]
\stackrel{\equref{Lbz1}}{=} - \left[\, \left[ x, \, y \right], \,
\left[ t , \, z \right] \, \right].
$$
We shall denote by ${\rm Aut}_{\rm Lbz} (\mathfrak{g})$ (resp.
${\rm Aut}_{\rm Lie} (\mathfrak{g})$) the automorphisms group of a
Leibniz (resp. Lie) algebra $\mathfrak{g}$. Any vector space $V$
is a Leibniz algebra with the trivial bracket $[x,\, y] = 0$, for
all $x$, $y\in V$ -- such a Leibniz algebra is called
\emph{abelian} and will be denoted by $V_0$. For two subspaces $A$
and $B$ of a Leibniz algebra $\mathfrak{g}$ we denote by $[A, \,
B]$ the vector space generated by all brackets $[a, \, b]$, for
any $a \in A$ and $b\in B$. In particular, $\mathfrak{g}' :=
[\mathfrak{g}, \, \mathfrak{g}]$ is called the derived subalgebra
of $\mathfrak{g}$. A Leibniz or a Lie algebra $\mathfrak{g}$ is
called \emph{metabelian} if $\mathfrak{g}'$ is an abelian
subalgebra of $\mathfrak{g}$, i.e. $[ \, [\mathfrak{g}, \,
\mathfrak{g}], \, [ \mathfrak{g}, \, \mathfrak{g} ] \, ] = 0$.
Hence, a metabelian Leibniz/Lie algebra is just a solvable algebra
of length $2$: this definition is exactly the counterpart of the
concept of a metabelian group used for Lie algebras in \cite{Dang,
dani, poro} and for Leibniz algebras in \cite{Dre}. We mention
that a more restrictive definition appeared in \cite{gauger, Gali}
where the term 'metabelian' is used to denote a $2$-step nilpotent
Lie algebra, i.e. $[ \mathfrak{g}, \, [\mathfrak{g}, \,
\mathfrak{g}] \, ] = 0$. Any $2$-step nilpotent Lie algebra is
metabelian; however, the Lie algebra $\mathfrak{b} (2, k)$ of
$2\times 2$ upper triangular matrices (non-zero diagonals are
permitted) is metabelian but not nilpotent. A partial converse is
proved in \cite[Lemma 3.3]{delbar}: any metabelian Lie algebra
with an ad-invariant metric form is a $3$-step nilpotent Lie
algebra. A subspace $\mathfrak{h}$ of a Leibniz algebra
$\mathfrak{g}$ is called a \emph{two-sided ideal} of
$\mathfrak{g}$  if $[h, \, g]$ and $[g, \, h] \in \mathfrak{h}$,
for all $h \in \mathfrak{h}$ and $g \in \mathfrak{g}$. In the
proof of \thref{itoLbz} we use intensively the Leibniz law as
given in \equref{Lbz1} as well as the following equivalent form of
it:
\begin{equation}\eqlabel{Lbz2}
\left[\left[x,\, z \right],\, y\right] = \left[\left[x, \,
y\right],\, z\right] - \left[x, \, \left[y, \, z\right]\right].
\end{equation}

The next result holds for Leibniz algebras defined over any
commutative ring $k$.

\begin{theorem}\thlabel{itoLbz}
Let $\mathfrak{g}$ be a Leibniz algebra and $A$, $B$ two abelian
subalgebras of $\mathfrak{g}$. Then any two-sided ideal
$\mathfrak{h}$ of $\mathfrak{g}$ contained in $A + B$ is a
metabelian Leibniz subalgebra of $\mathfrak{g}$.
\end{theorem}

\begin{proof}
We have to prove that $\left[ \, \left[x_1, \, x_2\right], \,
\left[x_3, \, x_4\right] \, \right] = 0$, for all $x_1$, $x_2$,
$x_3$, $x_4 \in \mathfrak{h}$. Since $\mathfrak{h} \subseteq A +
B$ we can find $a_{i} \in A$, $b_{i} \in B$ such that $x_i = a_i +
b_i$, for all $i = 1, \cdots, 4$. Using that $A$ and $B$ are both
abelian subalgebras of $\mathfrak{g}$ we obtain:
\begin{eqnarray*}
&& \left[ \, \left[x_1, \, x_2\right], \, \left[x_3, \, x_4\right]
\, \right] = \left[ \, \left[a_{1} + b_{1}, \, a_{2} +
b_{2}\right], \, \left[a_{3}
+ b_{3}, \, a_{4} + b_{4}\right] \, \right] \\
&& = \left[ \, \left[ a_{1}, \, b_{2}\right] + \left[b_{1}, \,
a_{2}\right],\, \left[a_{3},\, b_{4}\right] + \left[b_{3},\,
a_{4}\right] \, \right]\\
&& = \left[ \, \left[ a_{1}, \, b_{2}\right],\, \left[a_{3},\,
b_{4}\right] \, \right] + \left[ \, \left[ a_{1}, \,
b_{2}\right],\, \left[b_{3},\, a_{4}\right] \, \right] + \left[ \,
\left[b_{1}, \, a_{2}\right],\, \left[a_{3},\, b_{4}\right] \,
\right] + \left[ \, \left[b_{1}, \, a_{2}\right],\, \left[b_{3},\,
a_{4}\right] \, \right].
\end{eqnarray*}
The proof will be finished once we show that any member of the
last sum is equal to zero. Indeed, using that $A$ and $B$ are
abelian subalgebras and $\mathfrak{h}$ is a two-sided ideal of
$\mathfrak{g}$ we obtain: $[a_3, \, b_2] = [a_3, \, a_2 + b_2] =
[a_3, \, x_2] \in h$ and similarly $[a_1, \, b_4] = [x_1, \,
b_4]$, $[b_1, \, a_3] = [x_1, \, a_3]$ and $[b_4, \, a_2] = [b_4,
\, x_2]$ are all elements of $\mathfrak{h}$. Thus, we can find
some elements $a_{j} \in A$, $b_{j} \in B$, $j = 5, \cdots, 8$
such that:
\begin{eqnarray}
\left[a_3, \, b_2\right] = a_{5} + b_{5}, \quad \left[a_1, \, b_4
\right] = a_{6} + b_{6}, \quad \left[b_1, \, a_3\right] = a_{7} +
b_{7}, \quad \left[b_4, \, a_2\right] = a_{8} + b_{8}.
\eqlabel{Lbz3}
\end{eqnarray}
Then, using the fact that $A$ and $B$ are both abelian, we obtain:
\begin{eqnarray*}
\left[\, \left[a_1,\, b_2\right], \, \left[a_3, \, b_4 \right] \,
\right] &\stackrel{\equref{Lbz1}}{=}& \left[ \,
\underline{\left[\left[a_1, \, b_2\right], \, a_3\right]}, \, b_4
\right] - \left[ \, \underline{\left[\left[a_1,
\, b_2\right], \, b_4\right]}, \, a_3 \right]\\
&\stackrel{\equref{Lbz2}}{=}& \left[ \, \left[\left[a_1, \, a_3
\right], \, b_2\right] - \left[a_1, \, \left[a_3, \, b_{2}\right]
\, \right], \, b_4 \right] - \left[ \, \left[\left[a_1, \, b_4
\right], \, b_2\right] - \left[a_1, \,
\left[b_4, \, b_2\right] \, \right], \, a_3\right]\\
&=& - \left[ \, \left[a_1, \, \underline{\left[a_{3}, \,
b_2\right]} \, \right], \, b_4 \right] - \left[ \,
\left[\underline{\left[a_1, \, b_4\right]}, \,
b_2\right], \, a_3\right]\\
&\stackrel{\equref{Lbz3}}{=}& - \left[ \, \left[a_1, \, a_{5} +
b_{5} \right], \, b_4 \right] - \left[ \, \left[a_{6} +
b_{6}, \, b_2 \right], \, a_3 \right]\\
&=& - \underline{\left[\left[a_{1}, \, b_{5}\right], \, b_4
\right]} - \underline{\left[\left[a_{6},
\, b_2\right], \, a_{3}\right]}\\
&\stackrel{\equref{Lbz2}}{=}& \left[a_1, \, \left[b_4, \,
b_{5}\right] \, \right] - \left[ \, \left[a_1, \, b_4 \right], \,
b_{5}\right] + \left[a_{6}, \left[a_3, \, b_2\right]\, \right]
- \left[\, \left[a_{6},\, a_3\right], \, b_2\right]\\
&=& - \left[\underline{\left[a_{1}, \, b_4\right]}, \,
b_{5}\right] + \left[a_{6}, \underline{\left[a_{3}, \,
b_2\right]} \, \right]\\
&\stackrel{\equref{Lbz3}}{=}& - \left[a_{6} + b_{6}, \,
b_{5}\right] + \left[a_{6}, a_{5} + b_{5}\right] = - \left[a_{6},
\, b_{5}\right] + \left[a_{6}, b_{5}\right] = 0.
\end{eqnarray*}
Thus, $\left[\, \left[a_1,\, b_2\right], \, \left[a_3, \, b_4
\right] \, \right]= 0$. A similar computation shows that $\left[\,
\left[ a_1, \, b_2\right], \, \left[a_4, \, b_3 \right] \, \right]
= 0$. This implies, using \equref{surpriza}, that $ \left[ \,
\left[a_1, \, b_2 \right], \, \left[b_3, \, a_4 \right] \, \right]
= - \left[\, \left[ a_1, \, b_2\right], \, \left[a_4, \, b_3
\right] \, \right] = 0$, that is the second term of the above sum
is also equal to zero. Furthermore, we have:
\begin{eqnarray*}
\left[ \, \left[b_1, \, a_2\right], \, \left[a_3, \, b_4 \right]
\, \right] &\stackrel{\equref{Lbz1}}{=}&\left[ \,
\underline{\left[\left[b_1, \, a_2\right], \, a_3\right]}, \, b_4
\right] -\left[ \, \underline{\left[\left[b_1,
\, a_2\right], \, b_4\right]}, \, a_3\right]\\
&\stackrel{\equref{Lbz2}}{=}&\left[ \, \left[\left[b_1, \, a_3
\right],\, a_2\right] - \left[b_1, \, \left[a_3, \, a_2\right] \,
\right],\, b_4 \right] - \left[ \, \left[ \, \left[b_1, \, b_4
\right],\, a_2 \right] - \left[b_1,
\, \left[b_4, \, a_2\right] \, \right],\, a_3\right]\\
&=& \left[ \, \left[\underline{\left[b_1, \,  a_3\right]},\, a_2
\right],\, b_4\right] + \left[ \, \left[b_1, \,
\underline{\left[b_4, \, a_2\right]} \, \right],\, a_3 \right]\\
&\stackrel{\equref{Lbz3}}{=}& \left[ \, \left[a_{7} + b_{7},\, a_2
\right], \, b_4 \right] + \left[ \, \left[b_1, \,
a_{8} + b_{8}\right], \, a_3 \right]\\
&=& \underline{\left[\left[b_{7},\, a_2\right],\,
b_4\right]} + \underline{\left[\left[b_1, \, a_{8}\right],\, a_3\right]}\\
&\stackrel{\equref{Lbz2}}{=}& \left[ \, \left[b_{7}, \, b_4
\right], \, a_2 \right] - \left[b_{7}, \, \left[b_4, \, a_2\right]
\, \right] + \left[ \, \left[b_1, \, a_3 \right], \, a_{8} \right]
- \left[ b_1, \,
\left[a_3 , \, a_{8}\right] \, \right]\\
&=&  - \left[b_{7}, \, \underline{\left[b_4, \, a_2\right]} \,
\right] + \left[\underline{\left[b_1, \,
a_3\right]}, \, a_{8}\right]\\
&\stackrel{\equref{Lbz3}}{=}&  - \left[b_{7}, \, a_{8} +
b_{8}\right] + \left[a_{7} + b_{7}, \, a_{8}\right] = -
\left[b_{7}, \, a_{8}\right] + \left[b_{7}, \, a_{8}\right] = 0
\end{eqnarray*}
i.e. $\left[ \, \left[b_1, \, a_2\right], \, \left[a_3, \, b_4
\right] \, \right] = 0$, that is the third term of the above sum
is also equal to zero. A similar computation will prove that
$\left[ \, \left[b_1,\, a_2\right], \, \left[a_4, \, b_3 \right]
\, \right] = 0$. Finally, applying once again \equref{surpriza},
we obtain that $\left[ \, \left[b_1,\, a_2 \right], \, \left[b_3,
\, a_4\right] \, \right] = - \left[ \, \left[b_1,\, a_2\right], \,
\left[a_4, \, b_3 \right] \, \right] = 0$ and the proof is now
finished.
\end{proof}

As a special case of \thref{itoLbz} we obtain:

\begin{corollary}\colabel{itoLie} \textbf{(It\^{o}'s theorem for Leibniz algebras)}
Let $\mathfrak{g}$ be a Leibniz algebra over a commutative ring
$k$ such that $\mathfrak{g} = A + B$, for two abelian subalgebras
$A$ and $B$ of $\mathfrak{g}$. Then $\mathfrak{g}$ is metabelian.
\end{corollary}

\begin{example} \exlabel{bicrossed}
As an application of \coref{itoLie} we obtain that any
\emph{bicrossed product} $V_0 \bowtie P_0$ associated to a
\emph{matched pair} $(V_0, \, P_0 , \, \triangleleft, \,
\triangleright, \, \leftharpoonup, \, \rightharpoonup)$ of two
abelian Leibniz algebras $V_0$ and $P_0$ is a metabelian Leibniz
algebra. For the technical axioms defining a matched pair of
Leibniz algebras and the construction of the associated bicrossed
product we refer to \cite[Definition 4.5]{am-2013b}.
\end{example}

\begin{example}\exlabel{nugener}
It\^{o}'s theorem (\coref{itoLie}) can not be generalized in the
following direction: if $\mathfrak{g}$ is a Lie algebra such that
$\mathfrak{g} = A + B$, where $A$ is a metabelian subalgebra and
$B$ an abelian subalgebra, then $\mathfrak{g}$ is not necessarily
a metabelian Lie algebra. Indeed, let $\mathfrak{g}$ be the
$4$-dimensional Lie algebra (denoted by $L_5$ in \cite{ag}) having
the basis $\{e_1, \cdots, e_4\}$ and the bracket
$$
\left[ e_1, \, e_2\right] = e_2, \quad \left[ e_1, \, e_3\right] =
e_3, \quad \left[ e_1, \, e_4 \right] = 2\, e_4, \quad   \left[
e_2, \, e_3\right] = e_4.
$$
Then, $[ \, [ e_1, \, e_2 ], \, [ e_1, \, e_3 ] \, ] = e_4$, i.e.
$\mathfrak{g}$ is not metabelian. On the other hand, we have the
decomposition $\mathfrak{g} = A + B$, where $A$ is the metabelian
subalgebra of $\mathfrak{g}$ having $\{e_1, \, e_2\}$ as a basis
and $B$ is the abelian subalgebra of $\mathfrak{g}$ with the basis
$\{e_3, \, e_4\}$.
\end{example}

We end the section by looking at that the converse of
\coref{itoLie}. More precisely, we prove by counterexamples that
the converse of It\^{o}'s theorem fails to be true for both Lie
algebras and Leibniz algebras.

\begin{example}\exlabel{3dimlbz}
Let $\mathfrak{g}$ be the $3$-dimensional metabelian Leibniz
algebra over the field $\RR$ having the basis $\{e_1, \, e_2, \,
e_3\}$ and the bracket:
$$
[e_{2}, \, e_{2}] = e_{1}, \qquad [e_{3}, \, e_{3}] = e_{1}.
$$
We will prove that $\mathfrak{g}$ can not be written as a sum of
two abelian subalgebras. Indeed, suppose that $A$ is an abelian
subalgebra of dimension $1$ with basis $x = a e_{1} + b e_{2} + c
e_{3}$, $a$, $b$, $c \in k$. As $A$ is abelian we have $[x,\,x] =
0$ and this implies $b^{2} + c^{2} = 0$. Therefore, $b = c = 0$
and thus $A$ is generated by $e_{1}$. Assume now that $B$ is an
abelian subalgebra of dimension $2$ with basis $\{x = a_{1} e_{1}
+ a_{2} e_{2} + a_{3} e_{3}, \, y = b_{1} e_{1} + b_{2} e_{2} +
b_{3} e_{3}\}$. Again by the fact that $B$ is abelian we obtain
$a_{2} = a_{3} = b_{2} = b_{3} = 0$. Hence $\mathfrak{g}$ has no
abelian subalgebras of dimension $2$ and the conclusion follows.
\end{example}

Over the complex field $\CC$, the metabelian Lie algebra of
smallest dimension which can not be written as a sum of two proper
abelian subalgebras has dimension $5$. Indeed, if we consider the
classification of $3$ and $4$-dimensional complex Lie algebras
(see for instance \cite[Table 1 and 2]{ag}) one can easily see
that any metabelian Lie algebra in those lists can be written as a
direct sum of two proper abelian subalgebras.

\begin{example}\exlabel{5dimlie}
Let $\mathfrak{g}$ be the $5$-dimensional metabelian Lie algebra
having the basis $\{e_1, \, e_2, \, e_3, \, e_4, \, e_5\}$ and the
bracket given by:
\begin{equation}\eqlabel{brackmin}
[e_{1}, \, e_{2}] = e_{3}, \quad [e_{1}, \, e_{3}] = e_{4}, \quad
[e_{2}, \, e_{3}] = e_{5}.
\end{equation}
We will prove that $\mathfrak{g}$ can not be written as a sum of
two abelian subalgebras based on the following remark. Let $a =
\sum_{i=1}^{5} \, a_{i} \, e_{i}$ and $b = \sum_{i=1}^{5} \, b_{i}
\, e_{i}$, be two non-zero elements of an abelian Lie subalgebra
of $\mathfrak{g}$. Since $[a, \, b] = 0$, it follows from
\equref{brackmin}, that
\begin{equation}\eqlabel{aaa1}
a_{1} \, b_{2} = a_{2} \, b_{1}, \quad a_{1} \, b_{3} = a_{3} \,
b_{1}, \quad a_{2} \, b_{3} = a_{3} \, b_{2},
\end{equation}
that is, the rank of the following matrix
$$
\begin{pmatrix}
a_{1} & a_{2} & a_{3}\\
b_{1} & b_{2} & b_{3}
\end{pmatrix}
$$
is at most $1$. Based on this observation we will prove that ${\rm
dim_{k}} (A + B) \leq 4$, for any $A$, $B$ abelian subalgebras of
$\mathfrak{g}$. Assume that $\mathfrak{g}$ is the sum of two
abelian subalgebras, say, $\mathfrak{g} = A + B$. As
$\mathfrak{g}$ is not abelian we have ${\rm dim_{k}} (A) \geq 1$
and ${\rm dim_{k}} (B) \geq 1$. Let $l = {\rm dim_{k}} (A)$ and $t
= {\rm dim_{k}} (B)$ and consider $\mathcal{B}_{A} = \{X_{1},...,
\, X_{l}\}$, respectively $\mathcal{B}_{B} = \{Y_{1},..., \,
Y_{t}\}$, $k$-bases for $A$ and respectively $B$. As ${\rm
dim_{k}} (A + B) = 5$, we can find $5$ linearly independent
vectors among the generators $\{X_{1},..., \, X_{l}, \, Y_{1},...,
\, Y_{t}\}$ which form a basis of $A+B$. Again by the fact that
$\mathfrak{g}$ is not abelian, this basis of $A+B$ will contain
either three vectors from $\mathcal{B}_{A}$ and two vectors from
$\mathcal{B}_{B}$ or four vectors from $\mathcal{B}_{A}$ and one
vector from $\mathcal{B}_{B}$. Therefore, we are left to prove
that $\mathfrak{g}$ can not be written as the sum of two abelian
subalgebras of dimensions $3$ and $2$, respectively $4$ and $1$.
Suppose first that $\mathfrak{g} = A + B$, where $A$ and $B$ are
abelian subalgebras such that ${\rm dim_{k}} (A) = 3$ and ${\rm
dim_{k}} (B) = 2$. Consider $\mathcal{B}_{A} = \{X_{1} =
\sum_{i=1}^{5} \alpha_{i}\, e_{i},\, X_{2} = \sum_{i=1}^{5}
\beta_{i}\, e_{i}, \, X_{3} = \sum_{i=1}^{5} \gamma_{i}\, e_{i}\}$
and $\mathcal{B}_{B} =\{Y_{1} = \sum_{i=1}^{5} \delta_{i}\, e_{i},
\, Y_{2} = \sum_{i=1}^{5} \mu_{i} \, e_{i}\}$ $k$-bases of $A$,
respectively $B$, and let $V$ be the matrix of coefficients
\begin{equation}\eqlabel{matrV}
V = \begin{pmatrix} \alpha_{1} & \alpha_{2} & \alpha_{3} & \alpha_{4} & \alpha_{5} \\
\beta_{1} & \beta_{2} & \beta_{3} & \beta_{4} & \beta_{5}\\
\gamma_{1} & \gamma_{2} & \gamma_{3} & \gamma_{4} & \gamma_{5}\\
\delta_{1} & \delta_{2} & \delta_{3} & \delta_{4} & \delta_{5}\\
\mu_{1} & \mu_{2} & \mu_{3} & \mu_{4} & \mu_{5}
\end{pmatrix}.
\end{equation}
Let $W := V^{1, 2, 3}_{1, 2, 3}$ (resp. $T := V^{1, 2, 3}_{4, 5}$)
be the matrix formed with the first three columns and the first
three rows of $V$ (resp. the first three columns and the last two
rows of $V$). By applying \equref{aaa1} to the vectors in
$\mathcal{B}_{A}$, respectively $\mathcal{B}_{B}$, it follows that
${\rm rank}\, (W) = 1$ and ${\rm rank} \, (T) \leq 1$. We obtain
that ${\rm rank}\, (V)  \leq 4$, since ${\rm det}\, (V) = 0$. The
last assertion follows for instance by first expanding the
determinant of $V$ along the last column and then by expanding
again the resulting determinants along the last column. Thus
$\mathfrak{g}$ can not be written as a sum of abelian subalgebras
of dimensions $3$ and $2$. Finally, the second case will be
settled in the negative by proving that $\mathfrak{g}$ does not
have abelian subalgebras of dimension $4$. Suppose $A$ is such a
subalgebra with basis $\mathcal{B}_{A} = \{X_{1} = \sum_{i=1}^{5}
\alpha_{i}\, e_{i},\, X_{2} = \sum_{i=1}^{5} \beta_{i}\, e_{i}, \,
X_{3} = \sum_{i=1}^{5} \gamma_{i}\, e_{i}, \, X_{4} =
\sum_{i=1}^{5} \delta_{i}\, e_{i}\}$ and denote by $V'$ the matrix
consisting of the first four rows of $V$. We will reach a
contradiction by proving that ${\rm rank}\, (V') < 4$. By applying
\equref{aaa1} to the vectors in $\mathcal{B}_{A}$ we obtain that
${\rm rank}\, (V^{1, 2, 3}) = 1$, where $V^{1, 2, 3}$ is the
matrix consisting of the first three columns of $V'$. Hence we
obtain that ${\rm rank}\, (V') \leq 3$. This finishes the proof.
\end{example}

\section{On the structure and the classification of metabelian
Leibniz algebras}\selabel{class}

In this section a structure type theorem for metabelian Leibniz
algebras is proposed and, as an application, we give the
classification of all Leibniz algebras for which the derived
algebra has dimension $1$. First we will show that the classical
definition of a metabelian Leibniz algebra is equivalent to the
one recently introduced in \cite[Definition 2.8]{Mi2013}.

\begin{proposition}\prlabel{refordef}
A Leibniz algebra $\mathfrak{g}$ is metabelian if and only if
$\mathfrak{g}$ is an extension of an abelian Leibniz algebra by an
abelian Leibniz algebra, i.e. there exist two vector spaces $P$
and $V$ and an exact sequence of Leibniz algebra maps:
\begin{eqnarray} \eqlabel{primulsir}
\xymatrix{ 0 \ar[r] & V_0 \ar[r]^{i} & {\mathfrak{g}} \ar[r]^{\pi}
& P_0 \ar[r] & 0 }.
\end{eqnarray}
\end{proposition}

\begin{proof}
Any metabelian Leibniz algebra $\mathfrak{g}$ is an extension of
the abelian Leibniz algebra $(\mathfrak{g}/\mathfrak{g}')_0$ by
the abelian Leibniz algebra $\mathfrak{g}' = \mathfrak{g}'_0 $
since we have the following exact sequence of Leibniz algebra maps
\begin{eqnarray} \eqlabel{extenho2b}
\xymatrix{ 0 \ar[r] & \mathfrak{g}'_0 \ar[r]^{\iota} &
\mathfrak{g} \ar[r]^{p} & (\mathfrak{g}/\mathfrak{g}')_0 \ar[r] &
0 }
\end{eqnarray}
where $\iota$ is the inclusion map and $p$ the canonical
projection. Conversely, assume that we have an exact sequence
\equref{primulsir} of Leibniz algebra maps. By identifying $V_0
\cong i(V_0) \leq \mathfrak{g}$, we will assume that $V$ is a
subspace of $\mathfrak{g}$ and $i$ is just the inclusion map. The
sequence \equref{primulsir} being exact gives $[v, \, w] = 0$, for
all $v$, $w \in V$ and $[x, \, y] \in {\rm Ker} (\pi) = V$, for
all $x$, $y \in \mathfrak{g}$. Hence, we obtain $[\,[x, \, y ], \,
[z, \, t]\,] = 0$, for all $x$, $y$, $z$, $t\in \mathfrak{g}$,
i.e. $\mathfrak{g}$ is metabelian.
\end{proof}

Based on \prref{refordef}, the structure of metabelian Leibniz
algebras follows as a special case of \cite[Theorem 2.1 and
Corollary 2.10]{Mi2013}. However, for the reader's convenience we
will write down this result as it will be used in the proof of
\thref{clasific1}.

\begin{theorem}\thlabel{structura}
A Leibniz algebra $\mathfrak{g}$ is metabelian if and only if
there exists an isomorphism of Leibniz algebras $\mathfrak{g}
\cong V \, \#_{(\triangleleft, \, \triangleright, \, f)} \, P $,
for some vector spaces $P$ and $V$ connected by three bilinear
maps $\triangleleft: V \times P \to V$, $\triangleright : P \times
V \to V$ and $f: P \times P \to V$ satisfying the following
compatibilities for all $p$, $q$, $r\in P$ and $x \in V$:
\begin{eqnarray}
&& (x \triangleleft p) \triangleleft q = (x \triangleleft q)
\triangleleft p, \quad  p \triangleright (x \triangleleft q) = (p
\triangleright x) \triangleleft q = - \, p \triangleright (q \triangleright x) \eqlabel{discret1} \\
&& p \triangleright f(q, \, r) - f(p, \, q) \triangleleft r + f(p,
\, r) \triangleleft q = 0. \eqlabel{discret2}
\end{eqnarray}
The Leibniz algebra $V \, \#_{(\triangleleft, \, \triangleright,
\, f)} \, P$ is the vector space $V \times P$ with the bracket
given by:
\begin{eqnarray}
\{ (x, \, p), \, (y, \, q) \} := \bigl( x \triangleleft q + p
\triangleright y + f(p, \, q), \,\,  0 \bigl) \eqlabel{cotg2}
\end{eqnarray}
for all $x$, $y\in V$ and $p$, $q\in P$.
\end{theorem}

A system $(P, \, V, \, \triangleleft, \, \triangleright, \, f)$
satisfying the axioms \equref{discret1}-\equref{discret2} of
\thref{structura} will be called a \emph{Leibniz metabelian datum}
of $P$ by $V$ and we denote by ${\rm Met} \, (P, \, V)$ the set of
all metabelian datums $(P, \, V, \, \triangleleft, \,
\triangleright, \, f)$ of $P$ by $V$. The Leibniz algebra $V \,
\#_{(\triangleleft, \, \triangleright, \, f)} \, P$ is called the
\emph{metabelian product} associated to $(P, \, V, \,
\triangleleft, \, \triangleright, \, f)$.

For the sake of completeness we present the Lie algebra version of
\thref{structura} as it takes a simplified form - we use the
notations and terminology from \cite[Section 3]{am-2013a}.

\begin{corollary} \colabel{liemetabstr}
A Lie algebra $\mathfrak{g}$ is metabelian if and only if there
exists an isomorphism of Lie algebras $\mathfrak{g} \cong V \,
\natural_{(\triangleright, \, f)} \, P $, for some vector spaces $P$
and $V$ connected by two bilinear maps $\triangleright : P \times
V \to V$ and $f: P \times P \to V$ satisfying the following
compatibilities for any $p$, $q$, $r \in P$ and $x\in V$:
\begin{eqnarray}
p \triangleright (q \triangleright x) =  q\triangleright (p
\triangleright x), \quad f(p, p) = 0, \quad p\triangleright f(q,
r) + q\triangleright f(r, p) + r \triangleright f(p, q) = 0
\eqlabel{met2}
\end{eqnarray}
The Lie algebra $V \, \natural_{(\triangleright, \, f)} \, P$ is the
vector space $V \times P$ with the bracket defined by
\begin{eqnarray}
\left[ (x, \, p), \, (y, \, q) \right] := \bigl( p \triangleright
y - q \triangleright x + f(p, \, q), \,\,  0 \bigl)
\eqlabel{cotg2}
\end{eqnarray}
for all $x$, $y\in V$ and $p$, $q\in P$.
\end{corollary}

Let $P \neq 0$ be a vector space, $P^* = {\rm Hom}_k (P, \, k)$
the dual of $P$ and ${\rm Bil} \, (P \times P, \, k) \cong (P
\ot_k P)^* $ the space of all bilinear forms on $P$. On the direct
product of vector spaces $k\times P$ we construct three families
of Leibniz algebras as follows: for any non-trivial map $\lambda
\in P^* \setminus \{0\}$, we denote by $P^{\lambda}$ the non-Lie
Leibniz algebra having the bracket defined for any $a$, $b\in k$
and $p$, $q\in P$ by:
\begin{equation} \eqlabel{primafam}
\{ (a, \, p),  \, (b, \, q) \} := (a \lambda (q), \, 0).
\end{equation}
For any non-trivial bilinear form $f \in {\rm Bil} \, (P \times P,
\, k) \setminus \{0\}$, we denote by $P(f)$ the Leibniz algebra
with the bracket defined for any $a$, $b\in k$ and $p$, $q\in P$
by:
\begin{equation} \eqlabel{adouafam}
\{ (a, \, p),  \, (b, \, q) \} := (f(p, \, q), \, 0).
\end{equation}
We note that $P(f)$ is a Lie algebra if and only if
$f (p, p) = 0$, for all $p\in P$. Finally, for a pair $(\Lambda, \, f) \in (P^* \setminus \{0\}) \,
\times \, {\rm Bil} \, (P \times P, \, k)$ satisfying the following
compatibility condition for any $p$, $q$, $r\in P$
\begin{equation} \eqlabel{comp2}
\Lambda (p) \, f (q, \, r) + \Lambda (r) \, f (p, \, q) - \Lambda
(q) \, f (p, \, r) = 0
\end{equation}
we denote by $P_{(\Lambda, \, f)}$ the Lie algebra with the
bracket given by:
\begin{equation} \eqlabel{atreiaclas}
\{ (a, \, p),  \, (b, \, q) \} := ( - a \, \Lambda(q) + b \,
\Lambda(p) + f(p, \, q), \, 0)
\end{equation}
for all $a$, $b\in k$ and $p$, $q\in P$. $P_{(\Lambda, \, f)}$ is
indeed a Lie algebra since applying \equref{comp2} for $r = q$ and
taking into account that $\Lambda \in P^* \setminus \{0\}$ we
obtain $f(q, \, q) = 0$, for all $q\in P$. Thus, $\{ (a, \, p), \,
(a, \, p) \} = 0$, for all $a\in k$ and $p\in P$. These three
families of Leibniz algebras are all we need in order to provide
the complete classification of all metabelian Leibniz algebras
having the derived subalgebra of dimension $1$.

\begin{theorem} \thlabel{clasific1}
A Leibniz algebra $\mathfrak{g}$ has the derived
algebra $\mathfrak{g}'$ of dimension $1$ if and only if
$\mathfrak{g}$ is isomorphic to one of the Leibniz algebras
$P^{\lambda}$, $P (f)$ or $P_{(\Lambda, \, f)}$ from the three
families defined above, where $P \neq 0$ is a vector space,
$\lambda \in P^* \setminus \{0\}$, $f \in {\rm Bil} \, (P \times
P, \, k) \setminus \{0\}$ and $(\Lambda, \, f) \in (P^* \setminus
\{0\}) \times {\rm Bil} \, (P \times P, \, k)$ satisfying
\equref{comp2}. Furthermore, we have:

$(1)$ None of the Leibniz algebras from one family is isomorphic
to a Leibniz algebra pertaining to another family.

$(2)$ There exists an isomorphism of Leibniz algebras $P^{\lambda}
\cong P^{\lambda'}$, for all  $\lambda$, $\lambda' \in P^*
\setminus \{0\}$.

$(3)$ Two Leibniz algebras $P (f)$ and  $P (f')$ are isomorphic if
and only if there exists a pair $(u, \, \psi) \in k^* \times {\rm
Aut}_k (P)$ such that $u \, f(p, \, q) = f' (\psi(p), \, \psi
(q))$, for all $p$, $q\in P$.

$(4)$ Two Lie algebras $P_{(\Lambda, \, f)}$ and  $P_{(\Lambda',
\, f')}$ are isomorphic if and only if there exists a triple  $(v,
\, u, \, \psi) \in P^* \times k^* \times {\rm Aut}_k (P)$ such
that for any $p$, $q\in P$:
\begin{equation} \eqlabel{echivciud}
\Lambda = \Lambda' \circ \psi, \quad u \, f(p, \, q) = f' \bigl(
\psi(p), \, \psi (q) \bigl) \, + \, (\Lambda' \circ \psi ) (v(q) p - v(p) q)
\end{equation}
\end{theorem}

\begin{proof}
Using \prref{refordef} and \thref{structura}, we obtain that a
Leibniz algebra $\mathfrak{g}$ has the derived algebra
$\mathfrak{g}'$ of dimension $1$ (such a Leibniz algebra
$\mathfrak{g}$ is automatically non-abelian and metabelian) if and
only if there exists an isomorphism of Leibniz algebras
$\mathfrak{g} \cong k \, \#_{(\triangleleft, \, \triangleright, \,
f)} \, P $, for some vector space $P\neq 0$. Here we identify
$\mathfrak{g}' \cong k$ and $P$ will be chosen to be the quotient
vector space $\mathfrak{g}/\mathfrak{g}'$. Thus, we have to
compute the set ${\rm Met} \, (P, \, k)$ of all Leibniz metabelian
datums $(P, \, k, \, \triangleleft, \, \triangleright, \, f)$ of
$P$ by $k$ and then we have to describe and classify the
associated metabelian products $k \, \#_{(\triangleleft, \,
\triangleright, \, f)} \, P$. It is straightforward to show that
${\rm Met} \, (P, \, k)$ is in one-to-one correspondence with the
set of all triples $(\lambda, \, \Lambda, \, f) \in P^* \times P^*
\times {\rm Bil} \, (P \times P, \, k)$ satisfying the following
compatibilities for all $p$, $q$, $r\in P$:
\begin{equation} \eqlabel{echivciudgen}
\Lambda (p) \Lambda(q) = - \Lambda(p) \lambda (q), \qquad \Lambda
(p) \, f (q, \, r) - \lambda (r) \, f (p, \, q) + \lambda (q) \, f
(p, \, r) = 0.
\end{equation}
Through the above bijection, the metabelian datum $(P, \, k, \,
\triangleleft, \, \triangleright, \, f)$ corresponding to
$(\lambda, \, \Lambda, \, f) \in P^* \times P^* \times {\rm Bil}
\, (P \times P, \, k)$ is given for any $a\in k$ and $p\in P$ by:
\begin{equation} \eqlabel{transl}
a\triangleleft p := a\, \lambda (p), \quad p\triangleright a := a
\, \Lambda(p).
\end{equation}
The compatibility conditions \equref{echivciudgen} are exactly
those obtained by writing \equref{discret1}-\equref{discret2} via
the formulas \equref{transl}. We denote by $P_{(\lambda, \,
\Lambda, \, f)}$ the metabelian product $k \, \# \, P$ associated
to a triple $(\lambda, \, \Lambda, \, f) \in P^* \times P^* \times
{\rm Bil} \, (P \times P, \, k)$ satisfying \equref{echivciudgen}.
That is $P_{(\lambda, \, \Lambda, \, f)}$ is the vector space
$k\times P$ with the bracket given for any $a$, $b\in k$ and $p$,
$q\in P$ by:
\begin{equation} \eqlabel{general3}
\{ (a, \, p),  \, (b, \, q) \} := ( a \, \lambda(q) + b \,
\Lambda(p) + f(p, \, q), \, 0).
\end{equation}
If all the maps in the triple $(\lambda, \, \Lambda, \, f)$ are
trivial, then \equref{general3} shows that $P_{(\lambda, \,
\Lambda, \, f)}$ is abelian. Hence, we can consider only Leibniz
algebras $P_{(\lambda, \, \Lambda, \, f)}$ for which $(\lambda, \,
\Lambda, \, f) \neq (0, \, 0, \, 0)$. The next step will provide
necessary and sufficient conditions for two Leibniz algebras
$P_{(\lambda, \, \Lambda, \, f)}$ and $P_{(\lambda', \, \Lambda',
\, f')}$ to be isomorphic. We will prove a more general result
which will allow us to also compute the automorphisms group of
$P_{(\lambda, \, \Lambda, \, f)}$. More precisely, we shall prove
the following technical statement:

\emph{\textbf{Claim:}} Let $(\lambda, \, \Lambda, \, f) \neq (0,
\, 0, \, 0)$ and $(\lambda', \, \Lambda', \, f')$ be two triples
in $P^* \times P^* \times {\rm Bil} \, (P \times P, \, k)$
satisfying \equref{echivciudgen}. Then there exists a bijection
between the set of all morphisms of Leibniz algebras $\varphi:
P_{(\lambda, \, \Lambda, \, f)}  \to P_{(\lambda', \, \Lambda', \,
f')}$ and the set of all triples $(v, u, \psi) \in P^* \times k
\times {\rm End}_k(P)$ satisfying the following three
compatibilities for all $p$, $q\in P$:
\begin{eqnarray}
u \, \lambda &=& u \, \lambda' \circ \psi , \qquad u \,
\Lambda = u \, \Lambda' \circ \psi  \eqlabel{morf1} \\
u \, f(p, \, q) &=& f' \bigl( \psi(p), \, \psi (q) \bigl) \, + \,
v(p) \, \lambda'(\psi (q)) \, + \, v(q)\, \Lambda'(\psi (p)).
\eqlabel{morf2}
\end{eqnarray}
Through the above bijection, the Leibniz algebra map $\varphi =
\varphi_{(v, \, u, \, \psi)}$ corresponding to $(v, u, \psi) \in P^* \times k
\times {\rm End}_k(P)$ is given for any $(a, \, p) \in
k\times P$ by:
\begin{equation}\eqlabel{compmorfisme2}
\varphi: P_{(\lambda, \, \Lambda, \, f)}  \to P_{(\lambda', \,
\Lambda', \, f')}, \quad \varphi (a, \, p) = (a \, u + v(p), \,\,
\psi(p) ).
\end{equation}
Furthermore, $\varphi = \varphi_{(v, \, u, \, \psi)}$ is an
isomorphism if and only if $u \neq 0$ and $\psi \in {\rm Aut}_k
(P)$.

Indeed, first we notice that any linear map $\varphi: k \times P \to k
\times P$ is uniquely determined by a quadruple $(v, \, p_0, \, u,
\, \psi) \in P^* \times P \times k \times {\rm End}_k (P)$ such
that for any $a\in k$ and $p\in P$ we have:
$$
\varphi (a, \, p) = \varphi_{(v, \, p_0, \, u, \, \psi)} (a, \, p)
= (a\, u + v(p), \,\, a \,p_0 + \psi (p))
$$
We will prove that $\varphi = \varphi_{(v, \, p_0, \, u, \,
\psi)}: P_{(\lambda, \, \Lambda, \, f)}  \to P_{(\lambda', \,
\Lambda', \, f')}$ is a Leibniz algebra map if and only if $p_0 =
0$ and \equref{morf1}-\equref{morf2} hold. Indeed, if we write
\begin{equation} \eqlabel{morfismci}
\varphi \bigl( \{ (a, \, p), \,\, (b, \, q) \}  \bigl) = \{
\varphi (a, \, p), \,\, \varphi (b, \, q) \}
\end{equation}
for $(a, \, p) = (1, \, 0)$, $(b, \, q) = (1, \, 0)$ and
respectively $a = b = 0$ we obtain
$$
\lambda (q) \, p_0 = \Lambda(p)\, p_0 = f(p, \, q) \, p_0 = 0
$$
for all $p$, $q\in P$. It follows from here that $p_0 = 0$ since
$(\lambda, \, \Lambda, \, f) \neq (0, \, 0, \, 0)$. Hence,
$\varphi$ has the form given by \equref{compmorfisme2}. Taking
this into account, we can see that \equref{morfismci} holds for
$(a, p) = (1, 0)$ and $(b, q) = (0, q)$ (resp. $(a, p) = (0, p)$
and $(b, q) = (1, 0)$) if and only if the left hand side (resp.
right hand side) of the compatibility condition \equref{morf1}
holds. Finally, we note that \equref{morf2} is the same as
\equref{morfismci} applied for $(a, p) = (0, p)$ and $(b, q) = (0,
q)$. The fact that $\varphi_{(v, \, u, \, \psi)}$ is bijective if
and only if $u \neq 0$ and $\psi$ is bijective is straightforward
- we only mention that in this case the inverse of $\varphi$ is
given by $\varphi_{(v, \, u, \, \psi)}^{-1} = \varphi_{(
- u^{-1} \, v \circ \psi^{-1}, \, u^{-1}, \, \psi^{-1})}$.

In conclusion, we obtain that two Leibniz algebras $P_{(\lambda,
\, \Lambda, \, f)}$ and $P_{(\lambda', \, \Lambda', \, f')}$ are
isomorphic if and only if there exists a triple $(v, u, \psi) \in
P^* \times (k \setminus \{0\}) \times {\rm Aut}_k(P)$ satisfying
the following compatibilities for all $p$, $q\in P$:
\begin{eqnarray}
\lambda &=& \lambda' \circ \psi, \qquad \Lambda = \Lambda' \circ
\psi  \eqlabel{izomrofa} \\
u\, f(p, \, q) &=& f' \bigl( \psi(p), \, \psi (q) \bigl) \, + \,
v(p) \, \lambda'(\psi (q)) \, + \, v(q)\, \Lambda'(\psi (p)).
\eqlabel{izomrofb}
\end{eqnarray}
With this result in hand we return to the compatibility conditions
\equref{echivciudgen}. The first compatibility condition of
\equref{echivciudgen} requires a discussion on whether $\Lambda$
is the trivial map or $\Lambda \neq 0$. If $\Lambda \neq 0$, we
obtain that $\lambda = - \Lambda$ and the last compatibility of
\equref{echivciudgen} becomes precisely \equref{comp2} from the
defining axioms of the Lie algebra $P_{(\Lambda, \, f)}$; this and
the above claim prove $(4)$. Now, for $\Lambda = 0$ we also need
to discuss two cases: if $\lambda = 0$ then \equref{echivciudgen}
holds for any bilinear map $f \in {\rm Bil} \, (P \times P, \, k)$
and the associated metabelian product $k\# P$ is just the Leibniz
algebra $P(f)$ from $(3)$. Finally, we are left to discuss the
last case: $\Lambda = 0$ and $\lambda \neq 0$. Then
\equref{echivciudgen} becomes $\lambda (r) f (p, \, q) = \lambda
(q) f (p, \, r)$, for all $p$, $q\in P$. Let $r_0\in P$ such that
$\lambda (r_0) \neq 0$. It follows that $f (p, \, q) = \lambda (q)
f(p, \, r_0) \lambda (r_0)^{-1}$, i.e. there exists a linear map
$\theta \in P^*$ such that $f(p, \, q) = \theta (p) \lambda(q)$,
for all $p$, $q\in P$. We denote by $P^{(\lambda, \, \theta)}$ the
associated metabelian product $k\#P$. Using,
\equref{izomrofa}-\equref{izomrofb} we can easily see that there
exists an isomorphism of Leibniz algebras $P^{(\lambda, \,
\theta)} \cong P^{(\lambda, \, 0)}$, for all $\lambda \in P^*
\setminus \{0\}$ and $\theta \in P^*$ (take $u := 1$, $\psi :=
{\rm Id}_P$ and $v := \theta$ in
\equref{izomrofa}-\equref{izomrofb}). Thus, we have obtained that
for $\Lambda = 0$ and $\lambda \neq 0$, the corresponding
metabelian product $P_{(\lambda, \, \Lambda, \, f)}$ is isomorphic
to the Leibniz algebra $P^{\lambda}$ from $(2)$. The statement in
$(2)$ follows from a basic fact: namely, for any non-zero linear
maps $\lambda$, $\lambda' \in P^*$, there exists $\psi \in {\rm
Aut}_k (P)$ such that $\lambda = \lambda' \circ \psi$. The proof
is finished once we observe that the statement in $(1)$ follows
from \equref{izomrofa} - \equref{izomrofb} and the way the Leibniz
algebras $P^{\lambda}$, $P (f)$ and respectively $P_{(\Lambda, \,
f)}$ have been constructed.
\end{proof}

As a bonus, the proof of \thref{clasific1} provides the
description of the automorphisms groups of the Leibniz algebras
$P^{\lambda}$, $P (f)$ and $P_{(\Lambda, \, f)}$. For any vector
space $P$, consider $(P^*, +)$ to be the underlying abelian group
of the dual $P^*$ while $k^*$ stands for the group of units of the
field $k$. Then we can easily prove that the set $P^* \times k^*
\times {\rm Aut}_{k} (P)$ has a group structure with the
multiplication given for any $(v, \, u, \, \psi)$ and $(v', \, u',
\, \psi') \in P^* \times k^* \times {\rm Aut}_k (P)$ by:
\begin{equation}\eqlabel{strcturgr}
(v, \, u, \,  \psi) \cdot (v', \, u', \,  \psi') := (u\,
v' + v\circ \psi', \, uu', \, \psi\circ \psi')
\end{equation}
the unit being $(0, \, 1_k, \, {\rm Id}_P)$ while the inverse of
any $(v, \, u, \,  \psi)$ is given by $(\, - u^{-1} \, v\circ
\psi^{-1}, \, u^{-1}, \, \psi^{-1})$. We can see that $(P^*, +)
\cong P^* \times \{1\} \times \{ {\rm Id}_P\}$ is a normal
subgroup of $(P^* \times k^* \times {\rm Aut}_{k} (P), \, \cdot )$
and the direct product $k^* \times {\rm Aut}_{k} (P) \cong \{0\}
\times k^* \times {\rm Aut}_{k} (P)$ is a subgroup of $(P^* \times
k^* \times {\rm Aut}_{k} (P), \, \cdot )$. Furthermore, since $(v,
\, u, \psi) = (v \circ \psi^{-1}, \, 1_k, \, {\rm Id}_P) \cdot (0,
\, u, \, \psi)$ we obtain a factorization $P^* \times k^* \times
{\rm Aut}_{k} (P) = P^* \cdot \bigl( k^* \times {\rm Aut}_{k} (P)
\bigl)$, which shows that $(P^* \times k^* \times {\rm Aut}_{k}
(P), \, \cdot )$ is a semidirect product $P^* \ltimes \bigl(k^*
\times {\rm Aut}_{k} (P) \bigl)$ of groups.

For any $\lambda \in P^*$ we denote by ${\rm Aut}_{\lambda} (P) :=
\{ \,\psi \in {\rm Aut}_{k} (P) \, | \, \lambda \circ \psi =
\lambda \, \}$ the subgroup of all $\lambda$-invariant
automorphisms of a vector space $P$. Directly from the proof of
\thref{clasific1} and taking into account formula $\varphi_{(v, \,
u, \, \psi)} \circ \varphi_{(v', \, u', \, \psi')} = \varphi_{(av'
+ v\circ \psi', \, uu', \,  \psi \circ \psi')}$, for all $(v, \,
u, \, \psi)$, $(v', \, u', \, \psi') \in P^* \times k^* \times
{\rm End}_k (P)$ we obtain:

\begin{corollary} \colabel{automr}
Let $P$ be a vector space. Then:

$(1)$ There exists an isomorphism of groups ${\rm Aut}_{\rm Lbz} (P^{\lambda}) \cong
k^* \times {\rm Aut}_{\lambda} (P)$, for all $\lambda \in P^* \setminus \{0\}$, where
the right hand side is a direct product of groups.

$(2)$ For any bilinear map $f \in {\rm Bil} \, (P \times P, \, k)
\setminus \{0\}$ there exists an isomorphism of groups
$$
{\rm Aut}_{\rm Lbz} \bigl( P(f)\bigl) \, \cong \{ \, (v, \, u, \,
\psi) \in P^* \times k^* \times {\rm Aut}_{k} (P) \, | \,\, u f(
-, \, -) = f \bigl(\psi(-) , \, \psi(-) \bigl) \, \}
$$
where the right hand side has a group structure via the
multiplication given by \equref{strcturgr}.

$(3)$ For any pair $(\Lambda, \, f) \in (P^* \setminus \{0\})
\times {\rm Bil} \, (P \times P, \, k)$ satisfying \equref{comp2}
there exists an isomorphism of groups $ {\rm Aut}_{\rm Lie} \bigl(
P_{(\Lambda, \, f) } \bigl) \, \cong \, {\mathcal G} (P \, |
\Lambda, f)$, where ${\mathcal G} (P \, | \, \Lambda, f)$ is the
set consisting of all triples $(v, \, u, \, \psi) \in P^* \times
k^* \times {\rm Aut}_{k} (P)$ which satisfy the following
compatibilities for any $p$, $q\in P$
$$
\Lambda = \Lambda \circ \psi, \quad u \, f(p, \, q) = f \bigl(
\psi(p), \, \psi (q) \bigl) \, + \, \Lambda \bigl( v(q) p - v(p) q
\bigl)
$$
with the group structure given by \equref{strcturgr}.
\end{corollary}

\section*{Acknowledgement}

The authors are grateful to Otto H. Kegel for his comments on a
previous version of the paper.

\end{document}